\documentclass[a4paper,reqno]{amsart}

\usepackage{amssymb}
\usepackage{latexsym}
\usepackage{amsmath}
\usepackage{amsthm}
\usepackage{euscript}
\usepackage{bbm}

      \def\dC{{\mathbb C}}

      \def\dR{{\mathbb R}}

\def\cD{{\mathcal D}}      \def\cF{{\mathcal F}}
      
      \def\cL{{\mathcal L}}
   \def\cN{{\mathcal N}}   \def\cO{{\mathcal O}}

\def\cal H{{\mathcal H}}

\def\R{\mathbb{R}}
\def\C{\mathbb{C}}

\def\ran{{\text{\rm ran\,}}}
\def\dom{{\text{\rm dom\,}}}

\def\phi{\varphi}

\def\eps{\varepsilon}

\DeclareMathOperator{\Res}{Res}

\DeclareMathOperator{\clac}{cl_{ac}}

\DeclareMathOperator{\supp}{supp}
\DeclareMathOperator{\cl}{cl}
\DeclareMathOperator{\Real}{Re}
\DeclareMathOperator{\Imag}{Im}
\DeclareMathOperator{\spann}{span}

\newtheorem{theorem}{Theorem}[section]
\newtheorem{proposition}[theorem]{Proposition}
\newtheorem{corollary}[theorem]{Corollary}
\newtheorem{lemma}[theorem]{Lemma}

\newtheorem{definition}[theorem]{Definition}

\theoremstyle{remark}
\newtheorem{remark}[theorem]{Remark}

\numberwithin{equation}{section}

\begin{document}

\title[Titchmarsh--Weyl theory for Schr\"odinger operators]{Titchmarsh--Weyl theory for Schr\"odinger operators on unbounded domains}

\author{Jussi Behrndt}
\author{Jonathan Rohleder}
\email{behrndt@tugraz.at \and rohleder@tugraz.at}

\address{Technische Universit\"at Graz,
Institut f\"ur Numerische Mathematik,
Steyrergasse 30,
8010~Graz,
Austria}


\begin{abstract}
In this paper it is proved that the complete spectral data of selfadjoint Schr\"o\-din\-ger operators 
on unbounded domains can be described with an associated 
Dirichlet-to-Neumann map. In particular, a characterization of the isolated and embedded eigenvalues, 
the corresponding eigenspaces, as well as the continuous and 
absolutely continuous spectrum in terms of the limiting behaviour of the Dirichlet-to-Neumann map is obtained. 
Furthermore, a sufficient criterion for the absence of singular continuous spectrum is provided.
The results are natural multidimensional analogs of 
classical facts from singular Sturm--Liouville theory. 
\end{abstract}


\maketitle

\section{Introduction}

The Titchmarsh--Weyl $m$-function associated with a Sturm--Liouville differential expression 
plays a fundamental role in the direct and inverse spectral theory of the corresponding ordinary differential operators. 
It was introduced by H.~Weyl in his famous work~\cite{W10} and was further studied by E.\ C.~Titchmarsh in \cite{T62}, 
who investigated the analytic nature of this function as well as its connection to the spectrum. 
For a one-dimensional Schr\"odinger differential expression 
$- \frac{d^2}{d x^2} + q$ on the half-line $(0, \infty)$ with a bounded, real valued potential $q$
the Titchmarsh--Weyl $m$-function $m (\cdot)$ may be defined as
\begin{align*}
 m (\lambda) f_\lambda (0) = f_\lambda' (0),\qquad \lambda\in\dC\setminus\dR,
\end{align*}
where $f_\lambda$ is the unique solution (up to scalar multiples) in $L^2 (0, \infty)$ of the equation $- f'' + q f = \lambda f$; equivalently
$m(\lambda)$ combines two fundamental solutions to a solution in $L^2 (0, \infty)$.
The prominent role of the function $\lambda \mapsto m (\lambda)$ in the direct and inverse spectral theory of the associated selfadjoint 
operators is due to the celebrated fact that the complete spectral data is encoded and can be recovered from the knowledge of
$m (\cdot)$; cf.~\cite{CE69,T62}. Therefore the Titchmarsh--Weyl $m$-function became an indispensable tool 
in the spectral analysis of Sturm--Liouville differential operators, as well as more general Hamiltonian and canonical systems; 
for a small selection from the vast number of contributions see, e.g.,~\cite{A57,BHSW11,BCEP99,DK99,GNP97,GP87,HS81,KST12,S84,S96} for 
direct spectral problems and~\cite{B01,B52,BW05,GS96,GS00-1,GS00,LW11,M86,S99} for inverse problems.

The aim of the present paper is to develop Titchmarsh--Weyl theory in the multidimensional
setting for partial differential operators. Our focus is on selfadjoint Schr\"odin\-ger operators 
on unbounded domains. In our main results we prove that the $\lambda$-dependent Dirichlet-to-Neumann map $M (\lambda)$ on the boundary of
the domain, as the natural multidimensional analog of the Titchmarsh--Weyl $m$-function, 
determines the spectrum of the selfadjoint Schr\"odinger operator $A  = - \Delta + q$ with a bounded, real valued potential $q$ 
and a Dirichlet boundary condition uniquely.
We obtain an explicit characterization of the isolated and embedded eigenvalues, the corresponding eigenspaces, and the continuous and 
absolutely continuous spectrum in terms of the limiting behaviour of the Dirichlet-to-Neumann map $M (\lambda)$ when $\lambda$ approaches the 
real axis, and we provide a 
sufficient criterion for the absence of singular continuous spectrum. For instance, we show that $\lambda$ is an eigenvalue of $A$ if and only if the strong limit
 \textup{s}-$\lim_{\eta \searrow 0} \eta M (\lambda + i \eta)$ is non-trivial.
Our main results Theorem~\ref{thm:eigenvalues}, \ref{ACtheorem} and \ref{SCtheorem} extend to other selfadjoint realizations with Neumann
and more general (nonlocal) Robin boundary conditions, and also remain valid for second order, formally symmetric, 
uniformly elliptic differential operators under appropriate assumptions on the coefficients. In order to avoid technical complications, 
in this paper we discuss only the
case of an exterior domain with a $C^2$-boundary. The results can be extended to Lipschitz domains and to domains with non-compact boundaries; 
cf. Remark~\ref{remark1}. We mention that for bounded domains matters simplify essentially: In that case the spectrum of $A$ is purely 
discrete and it is known that the poles of the function $M (\cdot)$ coincide with the eigenvalues of $A$, see, e.g.,~\cite{NSU88} and \cite{BR12}.

In the recent past there has been a strong interest in combining and applying modern techniques from operator theory to 
partial differential equations. In the context of Titchmarsh--Weyl theory for elliptic differential equations we point out the 
paper \cite{AP04} by W.\,O.~Amrein and D.\,B.~Pearson, where a typical convergence property for Titchmarsh--Weyl $m$-functions in 
the one-dimensional
situation was extended to a multidimensional setting. We also refer the reader to the classical works~\cite{G68,LM72,V52} and to the more recent
contributions \cite{AB09,BL07,BL11,BGW09,BMNW08,GM09,GM11,GBook09,G11,M04,P12,R09} for other aspects of Titchmarsh--Weyl theory and 
spectral theory of elliptic
differential operators.
However, to  the best of our knowledge no attempts were made so far to extend the well-known results on the characterization
of the spectrum  
of ordinary differential operators in terms of the
Titchmarsh--Weyl $m$-function to elliptic differential operators on unbounded domains. We fill this gap in the present paper
and provide the natural multidimensional analogs. We also mention that the results in this paper can be generalized and interpreted 
in the more abstract context of boundary triples
and their Weyl functions from extension and spectral theory of symmetric and selfadjoint operators; cf. \cite{BL07,BL11,BMN02,DHMS06,DM91,DM95}.

\section{Preliminaries}

Let $\Omega$ be an open subset of $\R^n$, $n \geq 2$, such that $\R^n \setminus \overline{\Omega}$ is bounded, nonempty, and has a $C^2$-boundary 
$\partial \Omega$; 
for more general settings see Remark~\ref{remark1}. With $H^s (\Omega)$ and $H^s (\partial \Omega)$ we denote the 
Sobolev spaces of the order $s > 0$ on $\Omega$ and $\partial \Omega$, respectively. Moreover, for $u \in H^2 (\Omega)$ we denote 
by $u |_{\partial \Omega} \in H^{3/2} (\partial \Omega)$ the trace 
and by $\partial_\nu u |_{\partial \Omega} \in H^{1/2} (\partial \Omega)$ the trace of the 
derivative with respect to the outer unit normal. 

Let $q : \Omega \to \R$ be a bounded, measurable function. As usual, we define the {\em Dirichlet operator} $A$ in $L^2 (\Omega)$ 
corresponding to the Schr\"odinger differential expression $- \Delta + q$ by
\begin{align}\label{Dirichlet}
 A u = - \Delta u + q u, \quad \dom A = \left\{ u \in H^2 (\Omega) : u |_{\partial \Omega} = 0 \right\}.
\end{align}
It is well known that $A$ is a selfadjoint operator in $L^2 (\Omega)$ and that the spectrum $\sigma (A)$ of $A$ is bounded from 
below and accumulates to $+ \infty$; cf.~\cite{EE87,E98,LM72}.

Let $\lambda$ belong to the resolvent set $\rho (A)$ of $A$ and define
\begin{align}\label{Nlambda}
 \cN_\lambda = \left\{ u \in H^2 (\Omega) : - \Delta u + q u = \lambda u \right\}.
\end{align}
In order to define the Dirichlet-to-Neumann map associated with the differential expression $- \Delta + q$ recall that for each 
$\lambda \in \rho (A)$ and each $g \in H^{3/2} (\partial \Omega)$ the boundary value problem
\begin{align}\label{bvp}
 - \Delta u + q u = \lambda u, \quad u |_{\partial \Omega} = g,
\end{align}
has a unique solution $u_\lambda \in H^2 (\Omega)$; this follows essentially from the surjectivity of the trace map $H^2 (\Omega) \ni u \mapsto u |_{\partial \Omega} \in H^{3/2} (\partial \Omega)$. Thus for $\lambda \in \rho (A)$ the {\em Poisson operator} $\gamma (\lambda)$ from $L^2 (\partial \Omega)$ to $L^2 (\Omega)$ given by
\begin{align}\label{Poisson}
 \gamma (\lambda)  g = u_\lambda, \quad \dom \gamma (\lambda) = H^{3/2} (\partial \Omega),
\end{align}
is well-defined, where $u_\lambda$ is the unique solution of~\eqref{bvp} in $H^2 (\Omega)$. We remark that 
$\ran \gamma (\lambda) = \cN_\lambda$ holds. 

\begin{definition}\label{def1}
For $\lambda \in \rho (A)$ the {\em Dirichlet-to-Neumann map} $M (\lambda)$ in $L^2 (\partial \Omega)$ is defined by
\begin{align}\label{DN}
 M (\lambda) g = \partial_\nu u_\lambda |_{\partial \Omega}, \quad \dom M (\lambda) = H^{3/2} (\partial \Omega),
\end{align}
where $u_\lambda$ is the unique solution of~\eqref{bvp} in $H^2 (\Omega)$.
\end{definition}

The following proposition is crucial for the proofs of the main results in the next section.

\begin{proposition}\label{simplicity}
The linear space
\begin{align}\label{defectdense}
 \spann \bigcup_{\lambda \in \C \setminus \R} \cN_\lambda
\end{align}
is dense in $L^2 (\Omega)$.
\end{proposition}

\begin{proof}
Let us denote by $\widetilde q$ the extension of the potential $q$ by zero to all of $\R^n$. Then
\begin{align*}
 \widetilde A u = - \Delta u + \widetilde q u, \quad \dom \widetilde A = H^2 (\R^n),
\end{align*}
is a selfadjoint operator in $L^2 (\R^n)$ which is semibounded from below by the essential infimum of $\widetilde q$. 
Without loss of generality we assume that the lower bound $\mu$ of $\widetilde A$ is positive; this can always be achieved 
by adding a constant, thereby not changing the linear space in \eqref{defectdense}. Choose a function $\widetilde v \in L^2 (\R^n)$ 
such that $\widetilde v |_\Omega = 0$, and define 
$$\widetilde u_{\lambda, \widetilde v} := (\widetilde A - \lambda)^{-1} \widetilde v,\qquad \lambda \in \C \setminus \R.$$ 
Then the restriction $u_{\lambda, \widetilde v}$ of $\widetilde u_{\lambda, \widetilde v}$ 
to $\Omega$ satisfies $u_{\lambda, \widetilde v} \in H^2 (\Omega)$ and $-\Delta u_{\lambda, \widetilde v} + q u_{\lambda, \widetilde v} = \lambda u_{\lambda, \widetilde v}$, 
thus $u_{\lambda, \widetilde v} \in \cN_\lambda$ for all $\lambda \in \C \setminus \R$. 

Let $u \in L^2 (\Omega)$ be orthogonal to $\cN_\lambda$ for all $\lambda \in \C \setminus \R$ and let $\widetilde u$ denote the extension by zero of $u$ to $\R^n$. Then, in particular, 
\begin{align*}
 0 = ( u, u_{\overline \lambda, \widetilde v} ) = \big( \widetilde u, (\widetilde A - \overline \lambda)^{-1} \widetilde v \big)_{L^2 (\R^n)} = \big((\widetilde A - \lambda)^{-1} \widetilde u, \widetilde v \big)_{L^2 (\R^n)}
\end{align*}
for all $\lambda \in \C \setminus \R$, where $(\cdot, \cdot)$ and $(\cdot, \cdot)_{L^2 (\R^n)}$ are the inner products in $L^2 (\Omega)$ and $L^2 (\R^n)$, respectively. Since this identity holds for an arbitrary $\widetilde v \in L^2 (\R^n)$ with $\widetilde v |_\Omega = 0$, it follows
\begin{align}\label{resVan}
 \big( \widetilde A - \lambda \big)^{-1} \widetilde u = 0 \quad \text{on} \quad \R^n \setminus \Omega
\end{align}
for all $\lambda \in \C \setminus \R$. 

Following an idea of \cite[Section 3]{BS01} we consider the semigroup $T (t) = e^{- t \widetilde A^{1/2}}$, $t \geq 0$, which is generated by the square root of the uniformly positive operator $\widetilde A$. Then $t \mapsto T (t) \widetilde u$ is twice differentiable and we have
\begin{align*}
 \frac{d^2}{d t^2} T (t) \widetilde u = \widetilde A T (t) \widetilde u
\end{align*}
for $t > 0$, which implies
\begin{align}\label{n+1Var}
 \Big( - \frac{\partial^2}{\partial t^2} - \sum_{j = 1}^n \frac{\partial^2}{\partial x_j^2} + \widetilde q (x) \Big) T (t) \widetilde u (x) = 0, \quad x \in \R^n, t > 0,
\end{align}
in the distributional sense. In particular, by elliptic regularity, $(x, t) \mapsto T (t) \widetilde u (x)$ belongs locally to $H^2$ on $\R^n \times (0, \infty)$. Moreover, Stone's formula for the spectral measure $E (\cdot)$ of $\widetilde A$ and~\eqref{resVan} yield that
\begin{align*}
 E ((a, b)) \widetilde u = \lim_{\eps \searrow 0} \frac{1}{2 \pi i} \int_a^b \Big( \big(\widetilde A - (y + i \eps) \big)^{-1} \widetilde u - \big(\widetilde A - (y - i \eps) \big)^{-1} \widetilde u \Big) d y
\end{align*}
vanishes on $\R^n \setminus \Omega$ for all $a < b$ such that $a, b$ are no eigenvalues of $\widetilde A$. Consequently we have
\begin{align*}
 T (t) \widetilde u = \int_\mu^\infty e^{- t \sqrt{\lambda}} d E (\lambda) \widetilde u = 0 \quad \text{on} \quad \R^n \setminus \Omega
\end{align*}
for each $t > 0$. Therefore the function $(x, t) \mapsto T (t) \widetilde u (x)$ vanishes on $\R^n \setminus \Omega \times (0, \infty)$. 
From this and~\eqref{n+1Var} it follows by a unique continuation argument that $T (t) \widetilde u (x) = 0$ for all 
$(x, t) \in \R^n \times (0, \infty)$; see, e.g.,~\cite[Theorem~XIII.63]{RS78}. Thus $T (t) \widetilde u$ vanishes identically 
on $\R^n$ for all $t > 0$ and, taking the limit $t \searrow 0$, we obtain $\widetilde u = 0$. This implies 
$u = 0$ and hence the linear space \eqref{defectdense} is dense in $L^2 (\Omega)$.
\end{proof}

\begin{remark}\label{generalization}
The proof of Proposition~\ref{simplicity} shows that also $\spann \bigcup_{\lambda \in D} \cN_\lambda$ is dense in $L^2 (\Omega)$ 
with $D = \{x + i y : x \in \R,\, 0 < |y| < \eps \}$ for an arbitrary $\eps > 0$. In fact, with the help of the identity theorem for holomorphic functions it can be shown 
that $\C \setminus \R$ in \eqref{defectdense} can even be replaced by an arbitrary subset of~$\rho (A)$ with an accumulation point in~$\rho (A)$.
\end{remark}

\begin{remark}
The statement of Proposition~\ref{simplicity} is equivalent to the fact that the symmetric restriction
\begin{align*}
 S u = - \Delta u + q u, \quad \dom S = \left\{u \in \dom A : \partial_\nu u |_{\partial \Omega} = 0 \right\},
\end{align*}
of the Dirichlet operator in $L^2 (\Omega)$ is {\em simple} or {\em completely non-selfadjoint}; 
cf.~\cite[Chapter~VII-81]{AG93} and~\cite{K49}. The same property is known to hold for the minimal operator realizations of 
certain ordinary differential expressions which are in the limit point case at one endpoint, see~\cite{G72}.
\end{remark}

\section{Titchmarsh--Weyl theory for Schr\"odinger operators:\\ A characterization of the Dirichlet spectrum}

In this section we show how the isolated and embedded eigenvalues as well as the continuous spectrum of the Dirichlet operator $A$ in~\eqref{Dirichlet}
can be recovered from the limiting behaviour of the Dirichlet-to-Neumann map $M (\lambda)$ in~\eqref{DN} when $\lambda$ 
approaches the real axis. Moreover, we characterize the absolutely continuous spectrum of $A$ and prove a criterion for 
the absence of singular continuous spectrum.

As a preparation we recall some statements on the Poisson operator $\gamma (\lambda)$ in~\eqref{Poisson}, the Dirichlet-to-Neumann 
map $M (\lambda)$, and their relation to the resolvent of $A$. Their proofs are similar to the proof of~\cite[Lemma~2.4]{BR12} and will be omitted.
We also mention that in more abstract settings analog formulas are well known, see~\cite{BL07,DM91}.

\begin{lemma}\label{lemma1}
Let $\lambda, \zeta \in \rho (A)$, let $\gamma(\lambda), \gamma (\zeta)$ be the Poisson operators in~\eqref{Poisson}, and let $M (\lambda), M (\zeta)$ be the Dirichlet-to-Neumann maps in~\eqref{DN}. Then the following assertions hold.
\begin{enumerate}
 \item $\gamma (\lambda)$ is a bounded, densely defined operator from $L^2 (\partial \Omega)$ to $L^2 (\Omega)$. Its adjoint $\gamma (\lambda)^* : L^2 (\Omega) \to L^2 (\partial \Omega)$ is given by
 \begin{align*}
  \gamma (\lambda)^* u = - \partial_\nu \left( (A - \overline \lambda)^{-1} u \right) |_{\partial \Omega}, \quad u \in L^2 (\Omega).
 \end{align*}
 \item The identity
 \begin{align*}
  \gamma (\lambda) = \left( I + (\lambda - \zeta) (A - \lambda)^{-1} \right) \gamma (\zeta)
 \end{align*}
 holds.
 \item The relation
 \begin{align*}
  (\overline \zeta - \lambda) \gamma (\zeta)^* \gamma (\lambda) g = M (\lambda) g - M (\zeta)^* g, \quad g \in H^{3/2} (\partial \Omega),
 \end{align*}
 holds and $M (\overline \lambda) \subset M (\lambda)^*$. 
 \item $M (\lambda)$ is a densely defined, unbounded operator in $L^2 (\partial \Omega)$ and satisfies
\begin{align}\label{Weylformula}
 M (\lambda) = \Real M (\zeta) - \gamma (\zeta)^* \big( (\lambda - \Real \zeta) + (\lambda - \zeta) (\lambda - \overline \zeta) (A - \lambda)^{-1} \big) \gamma (\zeta);
 \end{align}
in particular, the limit $\lim_{\eta \searrow 0} \eta M (\mu + i \eta) g$ exists in $L^2 (\partial \Omega)$ for all $\mu \in \R$ and all $g \in H^{3/2} (\partial \Omega)$.
\end{enumerate}
\end{lemma}

Observe that \eqref{Weylformula} also implies that the function $M (\cdot)$ is strongly 
analytic on $\rho (A)$. 
In the following we agree to say that the function $M (\cdot)$ can be {\it continued analytically into} $\lambda \in \R$ if and only 
if there exists an open neighborhood $\cO$ of $\lambda$ in~$\C$ such that the $L^2 (\partial \Omega)$-valued function $M (\cdot) g$ can be continued analytically to $\cO$ for all $g \in H^{3/2} (\partial \Omega)$. We say that $M (\cdot)$ has a {\it pole at} $\lambda$ if and only if there exists $g \in H^{3/2} (\partial \Omega)$ such 
that $M (\cdot) g$ has a pole at $\lambda$. The {\it residue of $M (\cdot)$ at} $\lambda$ is defined in the strong sense by
\begin{equation*}
 \left( \Res_\lambda M \right) g := \Res_\lambda (M (\cdot) g), \quad g \in H^{3/2} (\partial \Omega),
\end{equation*}
where $\Res_\lambda (M (\cdot) g)$ is the usual residue of the $L^2 (\partial \Omega)$-valued function $M (\cdot) g$ at $\lambda$.

In the next theorem we denote by \textup{s}-$\lim$ the strong limit of 
an operator-valued function. Moreover, we denote by $\sigma_{\rm p} (A)$ and $\sigma_{\rm c} (A)$ the set of eigenvalues and the 
continuous spectrum of $A$, respectively. The following theorem is the multidimensional analog of the main theorem in~\cite{CE69} and 
of~\cite[Theorem~2]{HS81}, where several ODE situations were considered; see also~\cite{T62}. The proof of item~(i) is partly inspired 
by abstract considerations in~\cite{DLS93}; the characterization of the isolated and embedded eigenvalues in the items~(ii) and~(iii) 
uses methods from the more abstract works~\cite{BLu07,L06}.

\begin{theorem}\label{thm:eigenvalues}
Let $A$ be the selfadjoint Dirichlet operator in~\eqref{Dirichlet} and let $M (\lambda)$ be the Dirichlet-to-Neumann map in~\eqref{DN}. For $\lambda \in \R$ the following assertions hold.
\begin{enumerate}
 \item $\lambda \in \rho (A)$ if and only if $M (\cdot)$ can be continued analytically into $\lambda$.
 \item $\lambda \in \sigma_{\rm p} (A)$ if and only if \textup{s}-$\lim_{\eta \searrow 0} \eta M (\lambda + i \eta) \neq 0$. 
If $\lambda$ is an eigenvalue with finite multiplicity then the mapping
 \begin{align}\label{eigenspaceFin}
  \tau : \ker (A - \lambda)  \to \Big\{ \lim_{\eta \searrow 0} \eta M (\lambda + i \eta) g : g \in H^{3/2} (\partial \Omega) \Big\}, 
\quad u \mapsto \partial_\nu u |_{\partial \Omega},
 \end{align}
 is bijective; if $\lambda$ is an eigenvalue with infinite multiplicity then the mapping
 \begin{align}\label{eigenspace}
  \tau : \ker (A - \lambda) \to \cl_\tau \Big\{ \lim_{\eta \searrow 0} \eta M (\lambda + i \eta) g : g \in H^{3/2} (\partial \Omega) \Big\}, 
\quad u \mapsto \partial_\nu u |_{\partial \Omega},
 \end{align}
 is bijective, where $\cl_\tau$ denotes the closure in the linear space $\ran \tau$, equipped with the norm in~$L^2 (\partial \Omega)$.
 \item $\lambda$ is an isolated eigenvalue of $A$ if and only if $\lambda$ is a pole of $M (\cdot)$. 
 If $\lambda$ is an eigenvalue with finite multiplicity then the mapping
 \begin{align}\label{eigenspaceIsolFin}
  \tau : \ker (A - \lambda) \to \ran \Res_\lambda M, \quad u \mapsto \partial_\nu u |_{\partial \Omega},
 \end{align}
 is bijective;  if $\lambda$ is an eigenvalue with infinite multiplicity then the mapping
 \begin{align}\label{eigenspaceIsol}
  \tau : \ker (A - \lambda) \to \cl_\tau (\ran \Res_\lambda M ), \quad u \mapsto \partial_\nu u |_{\partial \Omega},
 \end{align}
 is bijective with $\cl_\tau$ as in~\textup{(ii)}.
 \item $\lambda \in \sigma_{\rm c} (A)$ if and only if \textup{s}-$\lim_{\eta \searrow 0} \eta M (\lambda + i \eta) = 0$ and $M (\cdot)$ cannot be continued analytically into $\lambda$.
\end{enumerate}
\end{theorem}

\begin{proof}
(i) It follows from Lemma~\ref{lemma1}~(iv) that $M (\cdot) g$ is analytic on $\rho (A)$ for each $g \in H^{3/2} (\partial \Omega)$. In order to verify the other implication, note first that the identity
\begin{align}\label{longFormula}
 \gamma(\zeta)^* (A - z)^{-1} \gamma(\nu) = \frac{M (z)}{(z - \nu) (\overline \zeta - z)} + \frac{M (\overline \zeta)}{(z - \overline \zeta) (\overline \zeta - \nu)} - \frac{M (\nu)}{(z - \nu) (\overline \zeta - \nu)}
\end{align}
holds for $\zeta, \nu, z \in \rho (A)$ satisfying $z \neq \nu, z \neq \overline \zeta$, and $\nu \neq \overline \zeta$. Indeed, Lemma~\ref{lemma1}~(ii) together with the first statement in Lemma~\ref{lemma1}~(iii) implies
\begin{align*}
	 \gamma(\zeta)^* (A - z)^{-1} \gamma(\nu) = 
	\frac{1}{z - \nu} 
	\left( \frac{M(z) - M(\overline \zeta)}{\overline \zeta - z} - \frac{M(\nu) - M(\overline \zeta)}{\overline \zeta - \nu} \right),
\end{align*}
and an easy computation yields~\eqref{longFormula}. Let us assume that $M (\cdot)$ can be continued analytically to some $\lambda \in \R$, that is, there exists an open neighborhood $\cO$ of $\lambda$ such that $M (\cdot) g$ can be continued analytically to $\cO$ for each $g \in H^{3/2} (\partial \Omega)$. Choose $a, b \notin \sigma_{\rm p} (A)$ with $\lambda \in (a, b)$ and $[a, b] \subset \cO$. The spectral projection $E ( (a, b) )$ of $A$ corresponding to the interval $(a, b)$ is given by
\begin{align}\label{stone}
 E ( (a, b) ) = \lim_{\delta \searrow 0} \frac{1}{2 \pi i} \int_a^b \left( (A - (t + i \delta) )^{-1} - (A - (t - i \delta) )^{-1} \right) d t,
\end{align}
where the integral on the right-hand side converges in the strong sense. Let us fix $\nu \in \C \setminus \R$. 
From~\eqref{longFormula} and~\eqref{stone} we obtain
\begin{align}\label{zero}
 \bigl(E ( (a,b) ) \gamma (\nu) g, \gamma (\zeta) h\bigr) = 0
\end{align}
for all $g, h \in H^{3/2} (\partial \Omega)$ and all $\zeta \in \C\setminus\R$, 
$\zeta\not=\overline\nu$, since $( M (\cdot) g, h)$ admits an analytic continuation into $\cO$ for all 
$g, h \in H^{3/2} (\partial \Omega)$, where $(\cdot, \cdot)$ is used for both the inner products in $L^2 (\Omega)$ and 
$L^2 (\partial \Omega)$. By Proposition~\ref{simplicity} and Remark~\ref{generalization}
\begin{align*}
 \spann \big\{ \gamma (\zeta) h : \zeta \in \C\setminus\R,\, \zeta\not=\overline\nu,\, h \in H^{3/2} (\partial \Omega) \big\}
\end{align*}
is dense in $L^2 (\Omega)$, thus~\eqref{zero} implies $E ( (a,b) ) \gamma (\nu) g = 0$ for all $g \in H^{3/2} (\partial \Omega)$. 
Since $\nu$ was chosen arbitrarily in $\C \setminus \R$ another application of Proposition~\ref{simplicity} 
yields $E ( (a, b) ) = 0$. This implies $\lambda \in \rho (A)$.

(ii) We prove that the mapping $\tau$ in~\eqref{eigenspace} is bijective for all $\lambda \in \R$; from this it follows immediately that $\lambda$ 
is an eigenvalue of $A$ if and only if \textup{s}-$\lim_{\eta \searrow 0} \eta M (\lambda + i \eta) \neq 0$. Let us fix $\lambda \in \R$. We prove first that the restriction $\tau$ of the trace of the normal derivative 
to $\ker (A - \lambda)$ is injective. Let $u \in \ker (A - \lambda)$ with $\partial_\nu u |_{\partial \Omega} = 0$. Then, 
denoting the extensions by zero of $u$ and $q$ to all of $\R^n$ by $\widetilde u$ and $\widetilde q$, respectively, we have $\widetilde u \in H^2 (\R^n)$ and
\begin{align*}
 \left(- \Delta + \widetilde q - \lambda \right) \widetilde u = 0.
\end{align*}
By construction $\widetilde u$ vanishes on the open, nonempty set $\R^n \setminus \overline \Omega$. Hence unique continuation implies $\widetilde u = 0$; cf.~\cite[Theorem~XIII.63]{RS78}. Thus $u = 0$ and we have proved the injectivity of $\tau$.

In order to prove the surjectivity of $\tau$ note first that for each $\zeta \in \C\setminus\R$ and each $u \in \ker (A - \lambda)$ the identity
\begin{align*}
 \tau u & = \partial_\nu u |_{\partial \Omega} = \partial_\nu \big( (A - \overline \zeta)^{-1} (A - \overline \zeta) u \big) |_{\partial \Omega} = (\lambda - \overline \zeta) \partial_\nu \big( (A - \overline \zeta)^{-1} u \big) |_{\partial \Omega} \\
 & = (\overline \zeta - \lambda) \gamma (\zeta)^* u
\end{align*}
holds by Lemma~\ref{lemma1}~(i), where $\gamma (\zeta)$ is the Poisson operator in~\eqref{Poisson}; hence, 
\begin{align}\label{taugamma}
\ran \tau = \ran \bigl(\gamma (\zeta)^* \upharpoonright \ker (A - \lambda) \bigr),\quad \zeta \in \C\setminus\R.
\end{align} 
In order to prove that $\tau$ in~\eqref{eigenspace} is surjective, we set
\begin{align*}
 \cF_\lambda := \Big\{ \lim_{\eta \searrow 0} \eta M (\lambda + i \eta) g: g \in H^{3/2} (\partial \Omega) \Big\}
\end{align*}
and show that
\begin{align}\label{ranRes}
 \cF_\lambda \subset \ran \bigl( \gamma (\zeta)^* \upharpoonright \ker (A - \lambda) \bigr) \subset \overline{\cF_\lambda},\quad \zeta \in \C\setminus\R .
\end{align}
Let us fix some $\zeta \in \C\setminus\R$. If we denote by $P_\lambda = E(\{ \lambda \})$ the orthogonal projection in $L^2 (\Omega)$ onto $\ker (A - \lambda)$ then for $\nu \in \C \setminus \R$ and $g \in H^{3/2} (\partial \Omega)$ we have
\begin{align*}
 \big\| \bigl(\eta (A - (\lambda + i \eta))^{-1} \! \!& - i P_\lambda\bigr) \gamma (\nu) g \big\|^2 \! \! \\
 & = \!
\int_\R \left| \frac{\eta}{t - \lambda - i \eta} - i \mathbbm{1}_{ \{ \lambda \} } (t) \right|^2 \! d \!\left( E (t)  \gamma (\nu) g,   \gamma (\nu) g\right)
\end{align*}
and hence the dominated convergence theorem yields
\begin{align*}
 \lim_{\eta \searrow 0} \eta  (A - (\lambda + i \eta))^{-1} \gamma (\nu) g = i P_\lambda \gamma (\nu) g.
\end{align*}
The formula~\eqref{longFormula} and the continuity of $\gamma (\zeta)^*$ imply
\begin{equation}\label{vereinfachung}
 \begin{split}
  \frac{\lim_{\eta \searrow 0} \eta M (\lambda + i \eta) g}{(\lambda - \nu) (\overline \zeta - \lambda)} & 
= \lim_{\eta \searrow 0} \eta \,\gamma (\zeta)^* (A - (\lambda + i \eta))^{-1} \gamma (\nu) g \\ &= i \gamma (\zeta)^* P_\lambda \gamma (\nu) g
 \end{split}
\end{equation}
for all $\nu \neq \overline \zeta$ and all $g \in H^{3/2} (\partial \Omega)$. Thus 
\begin{align}\label{FlambdaId}
 \cF_\lambda = \ran \big( \gamma (\zeta)^* \upharpoonright \spann \big\{ P_\lambda \gamma (\nu) g : \nu \in \C \setminus \R, \nu \neq \overline \zeta, g \in H^{3/2} (\partial \Omega) \big\} \big).
\end{align}
It follows from Proposition~\ref{simplicity} and Remark~\ref{generalization} that
\begin{align*}
 \spann \bigl\{ P_\lambda \gamma (\nu) g : \nu \in \C \setminus \R,\, \nu\not=\overline\zeta,\, g \in H^{3/2} (\partial \Omega) \bigr\}
\end{align*}
is dense in $\ker (A - \lambda)$, and, hence, from~\eqref{FlambdaId} and the continuity of $\gamma (\zeta)^*$ we obtain~\eqref{ranRes}. Furthermore, with \eqref{taugamma} we have $\cF_\lambda \subset \ran \tau\subset \overline{\cF_\lambda}$.
Since the closure $\cl_\tau (\cF_\lambda)$ of $\cF_\lambda$ in the normed space $\ran \tau$ 
(equipped with the norm of $L^2(\partial\Omega)$) coincides with the intersection of the closure $\overline \cF_\lambda$ (in $L^2(\partial\Omega)$) with 
$\ran \tau$, that is, $\cl_\tau (\cF_\lambda) = \overline \cF_\lambda \cap \ran \tau$, we conclude $\ran \tau=\cl_\tau (\cF_\lambda)$. Therefore $\tau$ is surjective and, hence, bijective. Clearly, if $\dim \ker (A - \lambda)$ is finite then equality holds in~\eqref{ranRes} which leads to the bijectivity of~\eqref{eigenspaceFin} and completes the proof of~(ii).

(iii) Let $\lambda$ be an isolated point of $\sigma (A)$. Then there exists an open neighborhood $\cO$ of $\lambda$ such that $z \mapsto (A - z)^{-1}$ is analytic on $\cO \setminus \{ \lambda \}$. Thus, by~(i), $M (\cdot)$ is analytic on $\cO \setminus \{ \lambda \}$ in the strong sense. Moreover, $\lambda \in \sigma_{\rm p} (A)$ and by~(ii) there exists $g \in H^{3/2} (\partial \Omega)$ such that $\lim_{\eta \searrow 0} i \eta M (\lambda + i \eta) g \neq 0$. Hence $\lambda$ is a pole of $M (\cdot)$ and it follows from~\eqref{Weylformula} and the corresponding property of the resolvent of $A$ that the order of the pole is one. Thus the limit 
\begin{align*}
 \lim_{z \to \lambda} (z - \lambda) M (z) g = \Res_\lambda M (\cdot) g
\end{align*}
exists for all $g \in H^{3/2} (\partial \Omega)$ and coincides with $\lim_{\eta \searrow 0} i \eta M (\lambda + i \eta) g$. Therefore~\eqref{eigenspaceIsol} is a consequence of~\eqref{eigenspace}. Analogously,~\eqref{eigenspaceIsolFin} follows from~\eqref{eigenspaceFin}. If, conversely, $\lambda$ is a pole of $M (\cdot)$ then there exists an open neighborhood $\cO$ of $\lambda$ such that $M (\cdot)$ is strongly analytic on $\cO \setminus \{ \lambda \}$ but not on $\cO$. Hence,~(i) implies $\lambda \in \sigma (A)$ and $\cO \setminus \{ \lambda \} \subset \rho (A)$; in particular, $\lambda$ is an eigenvalue of $A$. 

(iv) Since $\sigma_{\rm c} (A) = \C \setminus (\rho (A) \cup \sigma_{\rm p} (A) )$, the statement of~(iv) follows immediately from~(i) and~(ii).
\end{proof}

The next theorem shows how the absolutely continuous spectrum of the Dirichlet operator $A$ in~\eqref{Dirichlet} can be 
expressed in terms of the limits of the function $M (\cdot)$ towards real points. The result is well known in the one-dimensional 
setting for Sturm-Liouville
differential operators.
In a more abstract framework of extension theory of symmetric operators in Hilbert spaces and corresponding Weyl 
functions a similar result was proved in~\cite{BMN02}. We present a somewhat more direct proof avoiding the integral 
representation of a Nevanlinna function. We will make use of the following lemma, which can partly 
be found in, e.g., the monograph~\cite{T09}. Here, if $\mu$ is a finite Borel measure on $\R$, we denote the set of 
all growth points of $\mu$ by $\supp \mu$, that is,
\begin{align*}
 \supp \mu = \big\{ x \in \R : \mu ( (x - \eps, x + \eps) ) > 0~\text{for~all}~\eps > 0 \big\}.
\end{align*}
Moreover, for a Borel set $\chi \subset \R$ we define the {\em absolutely continuous closure} (also called {\em essential closure}) of $\chi$ by
\begin{align*}
 \clac (\chi) := \big\{ x \in \R : \left|(x - \eps, x + \eps) \cap \chi \right| > 0 ~\text{for~all}~\eps > 0 \big\},
\end{align*}
where $| \cdot |$ denotes the Lebesgue measure.

\begin{lemma}\label{mesLem}
Let $\mu$ be a finite Borel measure on $\R$ and denote by $F$ its Stieltjes transform, 
\begin{align*}
 F (\lambda) = \int_{\R} \frac{1}{t - \lambda} d \mu (t), \quad \lambda \in \C \setminus \R.
\end{align*}
Then the limit $\Imag F (x + i 0) = \lim_{y \searrow 0} \Imag F (x + i y)$ exists and is finite for Lebesgue almost all $x \in \R$. Let $\mu_{\rm ac}$ and $\mu_{\rm s}$ be the absolutely continous and singular part, respectively, of $\mu$ in the Lebesgue decomposition $\mu = \mu_{\rm ac} + \mu_{\rm s}$, and decompose $\mu_{\rm s}$ into the singular continuous part $\mu_{\rm sc}$ and the pure point part. Then the following assertions hold.
\begin{enumerate}
 \item $\supp \mu_{\rm ac} = \clac ( \{ x \in \R : 0 < \Imag F (x + i 0) < +\infty \} )$.
 \item The set $M_{\rm sc} = \{ x \in \R : \Imag F (x + i 0) = + \infty, \lim_{y \searrow 0} y F ( x + i y) = 0 \}$ is a support for $\mu_{\rm sc}$, that is, $\mu_{\rm sc} ( \R \setminus M_{\rm sc}) = 0$.
\end{enumerate}
\end{lemma}

\begin{proof}
The assertion on the existence of the limit $\Imag F (x + i 0)$ and item~(i) can be found in~\cite[Lemma~3.15 and Theorem~3.23]{T09}. In order to verify item~(ii) let us set
\begin{align*}
 (D \mu) (x) = \lim_{\eps \searrow 0} \frac{\mu ( (x - \eps, x + \eps) )}{2 \eps}
\end{align*}
for all $x \in \R$ such that the limit exists (finite or infinite). By~\cite[Theorem~A.38]{T09} the set 
$\left\{ x \in \R : (D \mu) (x) = +\infty \right\}$ is a support for $\mu_{\rm s}$ and $(D \mu) (x) = +\infty$ implies $\Imag F (x + i 0) = + \infty$, see~\cite[Theorem~3.23]{T09}. Consequently, also 
\begin{align*}
 \left\{ x \in \R : \Imag F (x + i 0) = +\infty \right\}
\end{align*}
is a support for $\mu_{\rm s}$. Moreover, note that $i \mu (\{x\} ) = \lim_{y \searrow 0} y F (x + i y)$ holds for all $x \in \R$; indeed, 
\begin{align*}
 \big| y F (x + i y) - i \mu ( \{x\} ) \big| \leq \int_{\R} \left| \frac{y}{t - (x + i y)} - i \mathbbm{1}_{\{x\}} (t) \right| d \mu (t) \to 0, \quad y \searrow 0,
\end{align*}
by the dominated convergence theorem. In particular, $\mu (\{x\}) = 0$ if and only if $\lim_{y \searrow 0} y F (x + i y) = 0$. Thus the claim of item~(ii) follows.
\end{proof}

Now the absolutely continuous spectrum of $A$ can be characterized in the same form as for ordinary differential operators.

\begin{theorem}\label{ACtheorem}
Let $A$ be the selfadjoint Dirichlet operator in~\eqref{Dirichlet} and let $M (\lambda)$ be the Dirichlet-to-Neumann map in~\eqref{DN}. Then the absolutely continuous spectrum of $A$ is given by
 \begin{align}\label{ACidentity}
  \sigma_{\rm ac} (A) = \overline{ \bigcup_{g \in H^{3/2} (\partial \Omega)} \clac \big( \big\{x \in \R : 0 < - \Imag (M (x + i 0) g, g) < 
+\infty \big\} \big) }.
 \end{align}
 In particular, if $a < b$ then $(a, b) \cap \sigma_{\rm ac} (A) = \emptyset$ if and only if for each $g \in H^{3/2} (\partial \Omega)$ one has $\Imag (M (x + i 0) g, g) = 0$ for almost all $x \in (a, b)$.
\end{theorem}

\begin{proof}
Let us set
\begin{align}\label{D}
 \cD := \left\{ \gamma (\zeta) g : g \in H^{3/2} (\partial \Omega), \zeta \in \C \setminus \R \right\} = 
\bigcup_{\zeta \in \C \setminus \R} \cN_\zeta,
\end{align}
where $\cN_\zeta$ is defined in~\eqref{Nlambda}. By Proposition~\ref{simplicity} $\spann \cD$ is dense in $L^2 (\Omega)$. We claim that 
the absolutely continuous spectrum of $A$ is given by 
\begin{align}\label{acid1}
 \sigma_{\rm ac} (A) = \overline{\bigcup_{u \in L^2(\Omega)} \supp \mu_{u, \rm ac}} = 
\overline{\bigcup_{\gamma (\zeta) g \in \cD} \supp \mu_{\gamma (\zeta) g, \rm ac}},
\end{align}
where $\mu_u := (E (\cdot) u, u)$ for $u \in L^2 (\Omega)$ and $E (\cdot)$ is the spectral measure of $A$. In fact, if $P_{\rm ac}$ denotes the orthogonal
projection onto the absolutely continuous subspace of $A$ then the absolutely continuous measures $\mu_{u, \rm ac}$ are given by
$$\mu_{u, \rm ac}=(E (\cdot) P_{\rm ac} u, P_{\rm ac} u)=\mu_{P_{\rm ac}u}.$$ 
Therefore, if $x\not\in\sigma_{\rm ac}(A)$ there exists $\varepsilon>0$ such that 
$E((x-\varepsilon,x+\varepsilon))P_{\rm ac}=0$ and hence $\mu_{u, \rm ac}((x-\varepsilon,x+\varepsilon))=0$ for all $u \in L^2(\Omega)$. This shows 
$(x-\varepsilon,x+\varepsilon)\cap\supp\mu_{u,{\rm ac}}=\emptyset$ for all $u\in L^2(\Omega)$ and hence
$$
x\not\in\overline{\bigcup_{u \in L^2(\Omega)} \supp \mu_{u, \rm ac}}.
$$
This yields the inclusions
\begin{equation*}
 \overline{\bigcup_{\gamma (\zeta) g \in \cD} \supp \mu_{\gamma (\zeta) g, \rm ac}}\subset
 \overline{\bigcup_{u \in L^2(\Omega)} \supp \mu_{u, \rm ac}}\subset \sigma_{\rm ac} (A).
\end{equation*}
Conversely, if $x$ does not belong to the right hand side of \eqref{acid1} then there exists $\varepsilon>0$ such that 
$(x-\varepsilon,x+\varepsilon)\subset \dR\setminus\supp\mu_{\gamma (\zeta) g, \rm ac}$ for all $\gamma (\zeta) g\in\cD$.
Thus 
\begin{equation*}
\Vert E((x-\varepsilon,x+\varepsilon))P_{\rm ac}\gamma (\zeta) g\Vert^2=\mu_{\gamma (\zeta) g, \rm ac}((x-\varepsilon,x+\varepsilon))=0
\end{equation*}
for all $\gamma (\zeta) g\in\cD$. Since $\spann \cD$ is dense in $L^2 (\Omega)$ by Proposition~\ref{simplicity} it follows that
$ E((x-\varepsilon,x+\varepsilon))P_{\rm ac} u=0$ holds for all $u\in L^2(\Omega)$, and hence $x\not\in\sigma_{\rm ac}(A)$. We have verified 
the identity \eqref{acid1}.

With the help of the 
formula~\eqref{Weylformula} we compute
\begin{align}\label{kette}
 \Imag (M & (x + i y) g, g) \nonumber \\
 & = - y \| \gamma (\zeta) g \|^2 - \left( |x - \zeta|^2 - y^2 \right) \Imag \left( (A - (x + i y) )^{-1} \gamma (\zeta) g, \gamma (\zeta) g \right) \nonumber \\
 & \quad  - 2 (x - \Real \zeta) y \Real \left( (A - (x + i y) )^{-1} \gamma (\zeta) g, \gamma (\zeta) g \right), 
\end{align}
for all $x \in \R$, $y>0$, $g \in H^{3/2} (\partial \Omega)$ and $\zeta \in \C \setminus \R$. Moreover,
\begin{align*}
 y \Real \left( (A - (x + i y) )^{-1} \gamma (\zeta) g, \gamma (\zeta) g \right) = \int_\R \frac{y (t - x)}{(t - x)^2 + y^2} 
d ( E (t) \gamma (\zeta) g, \gamma (\zeta) g )
\end{align*}
converges to zero as $y \searrow 0$ by the dominated convergence theorem. Therefore~\eqref{kette} implies
\begin{align}\label{MResId}
 \Imag ( M (x + i 0) g, g) = - |x - \zeta|^2 \Imag \left( (A - (x + i 0) )^{-1} \gamma (\zeta) g, \gamma (\zeta) g \right),
\end{align}
in particular, 
\begin{align}\label{ImagEqual}
 & \bigl\{ x \in \R : 0 < - \Imag (M (x + i 0) g, g) < + \infty \bigr\} \nonumber \\
 & \qquad \qquad \qquad = \bigl\{ x \in \R : 0 < \Imag \left((A - (x + i 0) )^{-1} \gamma (\zeta) g, \gamma (\zeta) g \right) < +\infty \bigr\}
\end{align}
holds for all $g \in H^{3/2} (\partial \Omega)$ and all $\zeta \in \C \setminus \R$. Note that the Stieltjes transform of 
the measure $\mu_{\gamma (\zeta) g}=(E(\cdot)\gamma (\zeta) g,\gamma (\zeta) g)$ is given by
\begin{align}\label{stern}
 F_{\gamma (\zeta) g} (x + i y) & = \int_\R \frac{1}{t - (x + i y)} d (E (t) \gamma (\zeta) g, \gamma (\zeta) g ) \nonumber \\
 & = \left( (A - (x + i y) )^{-1} \gamma (\zeta) g, \gamma (\zeta) g \right), \qquad x \in \R,\, y > 0.
\end{align}
Hence Lemma~\ref{mesLem}~(i) implies 
\begin{align*}
\supp \mu_{\gamma (\zeta) g, {\rm ac}} & = \clac \bigl( \bigl\{ x \in \R : 0 < \Imag F_{\gamma (\zeta) g} (x + i 0) < +\infty \bigr\} \bigr) \\
  &= \clac \bigl( \bigl\{ x \in \R : 0 < \Imag \left((A - (x + i 0) )^{-1} \gamma (\zeta) g, \gamma (\zeta) g \right) < +\infty \bigr\} \bigr)
 \end{align*}
and with the help of \eqref{ImagEqual} we conclude
\begin{align*}
\supp \mu_{\gamma (\zeta) g, {\rm ac}}  = \clac \bigl( \bigl\{ x \in \R : 0 < - \Imag \bigl(M (x + i 0) g, g \bigr) < +\infty \bigr\} \bigr).
 \end{align*}
Now the assertion \eqref{ACidentity} follows from \eqref{acid1}.

It remains to show that  $(a, b) \cap \sigma_{\rm ac} (A) = \emptyset$ if and only if for each $g \in H^{3/2} (\partial \Omega)$
one has $\Imag (M (x + i 0) g, g) = 0$ for almost all $x \in (a, b)$.
For abbreviation set
\begin{align*}
 M_{\rm ac} (g) := \bigl\{ x \in \R : 0 < - \Imag ( M (x + i 0) g, g)  < + \infty \bigr\}, \quad g \in H^{3/2} (\partial \Omega).
\end{align*}
If $(a, b) \cap \sigma_{\rm ac} (A) = \emptyset$ then $\emptyset  = \clac \big( M_{\rm ac} (g) \big) \cap (a, b)$ by~\eqref{ACidentity} for each $g \in H^{3/2} (\partial \Omega)$. Therefore, for each $g$ and each $x \in (a, b)$ there exists $\eps > 0$ such that 
\begin{align}\label{zeroset}
| (x - \eps, x + \eps) \cap M_{\rm ac} (g) | = 0.
\end{align}
It follows from~\eqref{MResId} and Lemma~\ref{mesLem} that $\Imag (M (x + i 0) g, g)$ exists and is finite for Lebesgue almost all $x \in \R$ and all $g \in H^{3/2} (\partial \Omega)$. Hence~\eqref{zeroset} implies $\Imag (M (x + i 0) g, g) = 0$ for all $g \in H^{3/2} (\partial \Omega)$ and almost all $x \in (a, b)$. The converse implication follows immediately from~\eqref{ACidentity}, since the absolutely continuous closure of a set of Lebesgue measure zero is empty.
\end{proof}

Next we formulate a sufficient criterion for the absence of singular continuous spectrum within some interval in 
terms of the limiting behaviour of the function~$M (\cdot)$. Again the one-dimensional counterpart for Sturm-Liouville operators is well known;
an abstract operator theoretic version is contained in \cite{BMN02}.

\begin{theorem}\label{SCtheorem}
Let $A$ be the selfadjoint Dirichlet operator in~\eqref{Dirichlet}, let $M (\lambda)$ be the Dirichlet-to-Neumann map in~\eqref{DN}, and let $a < b$. If for each $g \in H^{3/2} (\partial \Omega)$ there exist at most countably many $x \in (a, b)$ such that
 \begin{align}\label{SCspec}
  \Imag (M (x + i y) g, g) \to - \infty \quad \text{and} \quad y (M (x + i y) g, g) \to 0 \quad \text{as} \quad y \searrow 0
 \end{align}
 then $(a, b) \cap \sigma_{\rm sc} (A) = \emptyset$. 
\end{theorem}

\begin{proof}
As in the proof of Theorem~\ref{ACtheorem} one verifies the identity
\begin{align}\label{scid1}
 \sigma_{\rm sc} (A) = \overline{\bigcup_{\gamma (\zeta) g \in \cD} \supp \mu_{\gamma (\zeta) g, \rm sc}}
\end{align}
with $\cD$ defined in~\eqref{D} and $\mu_{\gamma (\zeta) g} = (E (\cdot) \gamma (\zeta) g, \gamma (\zeta) g)$. From~\eqref{SCspec} it follows with the help of~\eqref{vereinfachung} and~\eqref{MResId} that for each $g \in H^{3/2} (\partial \Omega)$ and each $\zeta \in \C \setminus \R$ there exist at most countably many $x \in (a, b)$ such that
\begin{align}\label{sEigensch}
 \Imag \left( (A - (x + i y) )^{-1} \gamma (\zeta) g, \gamma (\zeta) g \right) \to + \infty
\end{align}
and
\begin{align}\label{notppEigensch}
 y \left( (A - (x + i y) )^{-1} \gamma (\zeta) g, \gamma (\zeta) g \right) \to 0
\end{align}
as $y \searrow 0$. By Lemma~\ref{mesLem}~(ii) and~\eqref{stern} the set of those $x$ satisfying~\eqref{sEigensch} and~\eqref{notppEigensch} 
forms a support of $\mu_{\gamma (\zeta) g, \rm sc}$. It follows that $\mu_{\gamma (\zeta) g, \rm sc}$ has a countable support in $(a, b)$ for each $\gamma (\zeta) g \in \cD$. 
Since the measures $\mu_{\gamma (\zeta) g, \rm sc}$ do not have point masses, we have $(a, b) \cap \supp \mu_{\gamma (\zeta) g, \rm sc} = \emptyset$ 
for all $\gamma (\zeta) g \in \cD$ and, hence,~\eqref{scid1} yields $\sigma_{\rm sc} (A) \cap (a, b) = \emptyset$.
\end{proof}

As a corollary of the theorems of this section we provide sufficient criteria for the spectrum of the Dirichlet operator $A$ to 
be purely absolutely continuous or purely singularly continuous, respectively, in some interval.

\begin{corollary}
Let $A$ be the selfadjoint Dirichlet operator in~\eqref{Dirichlet}, let $M (\lambda)$ be the Dirichlet-to-Neumann map in~\eqref{DN}, and let $a < b$. Moreover, for all $x \in (a, b)$ let
\begin{align*}
 \textup{s-}\hspace{-1mm}\lim_{y \searrow 0} y M (x + i y) = 0.
\end{align*}
Then the following assertions hold.
\begin{enumerate}
\item If for each $g \in H^{3/2} (\partial \Omega)$ there exist at most countably many $x \in (a, b)$ such that 
$\Imag (M (x + i 0) g, g) = - \infty$ then $\sigma (A) \cap (a, b) = \sigma_{\rm ac} (A) \cap (a, b)$.
	\item If for each $g \in H^{3/2} (\partial \Omega)$ one has $\Imag ( M (x + i 0) g, g) = 0$ for almost all $x \in (a, b)$ then $\sigma (A) \cap (a, b) = \sigma_{\rm sc} (A) \cap (a, b)$.
	
\end{enumerate}
\end{corollary}

\begin{remark}\label{remark1}
The main results of the present paper, Theorem~\ref{thm:eigenvalues} as well as Theorem~\ref{ACtheorem} and Theorem~\ref{SCtheorem}, 
remain true when the Dirichlet operator $A$ is replaced by the selfadjoint operator $-\Delta+q$ in $L^2 (\Omega)$ subject to a Robin type boundary condition
\begin{align*}
 \Theta u |_{\partial \Omega} = \partial_\nu u |_{\partial \Omega},
\end{align*}
where $\Theta$ is a selfadjoint, bounded operator in $L^2 (\partial \Omega)$, and $M (\lambda)$ is replaced by the corresponding Robin-to-Dirichlet map $M_\Theta (\lambda) = (\Theta - M (\lambda))^{-1}$. Moreover, the results can be carried over to more general second order uniformly elliptic, formally symmetric differential expressions of the form
\begin{align*}
  \cL = - \sum_{j,k = 1}^n \partial_j a_{jk} \partial_k + \sum_{j = 1}^n \big( a_j \partial_j - \partial_j \overline a_j \big)  + a
\end{align*}
under suitable smoothness and boundedness conditions on the coefficients $a_{jk}$, $a_j$, $a$, $1 \leq j, k \leq n$, and to domains with less regular (e.g.\ Lipschitz) boundaries. 
Finally we remark that unbounded domains with 
non-compact (sufficiently regular) boundaries can be treated in almost the same way.
\end{remark}

\section*{Acknowledgement}
This research was supported by the Austrian Science Fund (FWF): Project P~25162-N26. J. Behrndt gratefully acknowledges the stimulating
atmosphere at the Isaac Newton Institute for Mathematical Sciences in Cambridge (UK)
in July and August 2012 where parts of this paper were written during the
research program {\it Spectral Theory of Relativistic Operators}. The authors wish to thank A.~Strohmaier who drew 
their attention to the article~\cite{BS01} and 
gave a very useful hint for the proof of Proposition~\ref{simplicity} in its present form.

\end{document}